\documentclass[3p]{elsarticle}

\usepackage{amsmath,amsthm,amssymb,mathrsfs}
\usepackage{thmtools, thm-restate}
\usepackage{enumerate}
\usepackage{slashed}
\usepackage{txfonts,dsfont}
\usepackage[breaklinks,colorlinks]{hyperref}
\usepackage{nameref}
\usepackage{cleveref}

\makeatletter
\newcommand{\autorefcheckize}[1]{%
  \expandafter\let\csname @@\string#1\endcsname#1%
  \expandafter\DeclareRobustCommand\csname relax\string#1\endcsname[1]{%
    \csname @@\string#1\endcsname{##1}\wrtusdrf{##1}}%
  \expandafter\let\expandafter#1\csname relax\string#1\endcsname
}
\makeatother

\declaretheorem[numberwithin=section]{theorem}
\declaretheorem[sibling=theorem, name=Lemma]{lem}
\declaretheorem[sibling=theorem, name=Corollary]{cor}

\declaretheorem[sibling=theorem, name=Remark]{rem}

\numberwithin{equation}{section}

\newtheorem{lemma}[theorem]{Lemma}
\newcommand{\norm}[1]{\left\lVert#1\right\rVert}
\newcommand{\abs}[1]{\left\lvert#1\right\rvert}
\newcommand{\set}[1]{\left\{#1\right\}}
\newcommand{\hin}[2]{\left\langle#1,#2\right\rangle}

\newcommand*{\To}{\longrightarrow}
\newcommand*{\Rmn}[1]{\uppercase\expandafter{\romannueral#1}}
\newcommand*{\dif}{\mathop{}\!\mathrm{d}}

\journal{XXX}

\allowdisplaybreaks

\begin{document}

\begin{frontmatter}

\title{Existence results for a generalized mean field equation  on a closed Riemann surface\tnoteref{swy}}

\author[whu1,whu2]{Linlin Sun\corref{sll}}
\address[whu1]{School of Mathematics and Statistics, Wuhan University, Wuhan 430072, China}
\address[whu2]{Hubei Key Laboratory of Computational Science, Wuhan University, Wuhan, 430072, China}
\ead{sunll@whu.edu.cn}

\author[ruc]{Yamin Wang}
\ead{2017100918@ruc.edu.cn}

\author[ruc]{Yunyan Yang}
\address[ruc]{Department of Mathematics, Renmin University of China, Beijing 100872, China}
\ead{yunyanyang@ruc.edu.cn}

\cortext[sll]{Corresponding author.}
\tnotetext[swy]{This work is partially supported by the National Natural Science Foundation of China (Grant Nos. 11971358, 11801420, 11721101), and by the National Key Research and Development Project SQ2020YFA070080.}

\begin{abstract}
Let $\Sigma$ be a closed Riemann surface, $h$ a positive smooth function on $\Sigma$, $\rho$ and $\alpha$ real numbers. In this paper, we study a generalized mean field equation
\begin{align*}
    -\Delta u=\rho\left(\dfrac{he^u}{\int_\Sigma he^u}-\dfrac{1}{\mathrm{Area}\left(\Sigma\right)}\right)+\alpha\left(u-\fint_{\Sigma}u\right),
\end{align*}
where $\Delta$ denotes the Laplace-Beltrami operator. We first derive a uniform bound for solutions when $\rho\in (8k\pi, 8(k+1)\pi)$ for some non-negative integer number $k\in \mathbb{N}$  and $\alpha\notin\mathrm{Spec}\left(-\Delta\right)\setminus\set{0}$. Then we obtain existence results for $\alpha<\lambda_1\left(\Sigma\right)$
by using the Leray-Schauder degree theory and the minimax method, where $\lambda_1\left(\Sigma\right)$ is the first positive eigenvalue for
$-\Delta$.
\end{abstract}

\begin{keyword}
generalized mean field equation \sep blow-up analysis \sep topological degree \sep minimax method

\MSC[2020]  53C21\sep 58J05\sep 35J20\sep 35J61
\end{keyword}

\end{frontmatter}


\section{Introduction}
Let $\Sigma$ be a closed Riemann surface with area one and $\Delta$ be the  Laplace-Beltrami operator.
The mean field equation is stated as follows
\begin{equation}\label{eq:mf}
-\Delta u=\rho\left(\dfrac{he^u}{\int_\Sigma he^u }-1\right),
\end{equation}
where $\rho$ is a real number and $h$ is a smooth function on $\Sigma$.  It comes from the prescribed Gaussian curvature problem  \cite{Ber71riemannian,KazWar74curvature,ChaYan87prescribing,ChaYan88conformal,ChaLiu93nirenberg,CheDin87scalar,Han90prescribing},
and also appears in various context such as the abelian Chern-Simons-Higgs models  \cite{RicTar00vortices,Nol03nontopological,NolTar00vortex,NolTar99double,JacWei90self-dual,HonKimPac90multivortex,CafYan95vortex,SprYan95topological,Tar96multiple}.

The existence of solutions of the mean field equation has been widely studied in recent decades.  Recall the strong Trudinger-Moser inequality \cite[Theorem 1.7]{Fon93sharp}
\begin{align}\label{eq:StrongTM}
    \sup_{u\in H^1\left(\Sigma\right), \int_{\Sigma}\abs{\nabla u}^2\leq 1, \int_{\Sigma}u=0}\int_{\Sigma}\exp\left(4\pi u^2\right)<\infty,
\end{align}
which implies the Trudinger-Moser inequality
\begin{align}\label{eq:TM}
    \ln\fint_{\Sigma}e^{u}\leq\dfrac{1}{16\pi}\int_{\Sigma}\abs{\nabla u}^2+\fint_{\Sigma}u+c,
\end{align}
where $c$ is a uniform constant depends only on the geometry of $\Sigma$. Consequently, for $\rho<8\pi$, the Trudinger-Moser inequality \eqref{eq:TM} gives a minimizer to the action functional
\begin{align*}
   H^1\left(\Sigma\right)\ni u\mapsto\dfrac{1}{2}\int_{\Sigma}\abs{\nabla u}^2+\rho\left(\fint_{\Sigma}u-\ln\abs{\int_{\Sigma}he^{u}}\right).
\end{align*}
Many partial existence results have been obtained for noncritical cases, see for examples Struwe and Tarantello \cite{StruTar98multivortex}, Ding, Jost, Li and Wang \cite{DinJosLiWan99existence}, Chen and Lin \cite{CheLin03topological}, Djadli \cite{Dja08existence} and the references therein.
For the first critical case $\rho=8\pi$ and $h>0$, several sufficient conditions for the existence to the Nirenberg problem (i.e., $\rho=8\pi$ and the genus is zero) were given by Moser \cite{Moser73nonlinear}, Aubin \cite{Aub79meilleures},  Chang and Yang \cite{ChaYan87prescribing,ChaYan88conformal}, Ji \cite{Ji04positive} and others; a sufficient condition for the minimizing solution to the mean field equation on a closed Riemann surface $\Sigma$ with positive genus was given by Ding, Jost, Li and Wang \cite{DinJosLiWan97differential};  it was also independently proved by  Nolasco and Tarantello \cite{NolTar98sharp} when $\Sigma$ is a flat torus. Chen and Lin \cite{CheLin02sharp} obtained an existence result for general critical cases (i.e., $\rho=8k\pi$) which generalized Ding, Jost, Li and Wang's result \cite{DinJosLiWan97differential}.  For sign-changing potential $h$, we refer the reader to \cite{MarLop16existence,MarLopRui18compactness,TerMarIan18prescribed,Han90prescribing,CheLi08priori,CheDin87scalar} and the references therein. For the uniqueness to \eqref{eq:mf} we refer the reader to Gui and Moradifam \cite{GuiMor18sphere}, Shi, Sun, Tian and Wei  \cite{ShiSunTiaWei19uniqueness} and others. We refer to \cite{Tar10analytical} for a nice survey on the mean field equation.

Among various improvements of (\ref{eq:StrongTM}), it was proved by Yang \cite{Yan15extremal} that for all $\alpha<\lambda_1(\Sigma)$,
the first positive eigenvalue of the negative Laplacian $-\Delta$, there holds
\begin{align}\label{eq:improvedTM}\sup_{u\in H^1(\Sigma),\int_\Sigma|\nabla u|^2-\alpha\int_\Sigma u^2\leq 1,\int_\Sigma u=0}
\int_\Sigma e^{4\pi u^2}<\infty.\end{align}
This leads to an analog of (\ref{eq:TM}), namely for any $\alpha<\lambda_1(\Sigma)$, there exists some constant $c$ depending only on the geometry of $\Sigma$ such that
for all $u\in H^1(\Sigma)$ with $\int_\Sigma u=0$,
\begin{align}\label{eq:TM-1}
    \ln\int_{\Sigma}e^{u}\leq\dfrac{1}{16\pi}\int_{\Sigma}\left(\abs{\nabla u}^2-\alpha u^2\right)+c.
\end{align}
In view of (\ref{eq:improvedTM}) and (\ref{eq:TM-1}), it is natural to consider the following generalized mean field equation
\begin{equation}\label{eq:gmf}
-\Delta u=\rho\left(\dfrac{he^u}{\int_\Sigma he^u }-1\right)+\alpha\left(u-\fint_{\Sigma}u\right),
\end{equation}
where $h$ is a smooth positive function on $\Sigma$ and $\rho, \alpha\in\mathbb{R}$. The related functional would be written as
\begin{align}\label{eq:functional}
    J_{\rho,\alpha}(u)=J_{\rho,\alpha, h}(u)=\dfrac{1}{2\rho}\int_{\Sigma}\left(\abs{\nabla u}^2-\alpha\left(u-\fint_{\Sigma}u\right)^2\right)+\fint_{\Sigma}u-\ln{\int_{\Sigma}he^{u}},\quad u\in H^1\left(\Sigma\right).
\end{align}
As an immediately consequence of \eqref{eq:TM-1}, there exists a solution to \eqref{eq:gmf} when $\rho<8\pi$ and $\alpha<\lambda_1\left(\Sigma\right)$. In \cite{YanZhu18existence}, Yang and Zhu gave a sufficient condition such that \eqref{eq:gmf} has a solution when $\rho=8\pi$ and $\alpha<\lambda_1\left(\Sigma\right)$. 

Our aim is to study the existence problem for \eqref{eq:gmf} when $\rho\not=8k\pi$ with $k\in \mathbb{N}$ and $\alpha<\lambda_1\left(\Sigma\right)$. To achieve this goal, we  begin by  studying the blow-up phenomena for the generalized mean field equation \eqref{eq:gmf}. For the classical case (i.e., $\alpha=0$), it is well known that the blow-up phenomena only occurs if the parameter $\rho$ is a multiple of $8\pi$. Therefore the set of solutions is compact when $\rho\in\mathbb{R}\setminus 8\pi \mathbb{N}^{*}$. Our first main theorem in the following can be viewed as an analogous conclusion.

\begin{theorem}\label{thm:compactness}Let $I, J, K$ be compact subsets with $ I\subset\mathbb{R}\setminus 8\pi\mathbb{N}^*, J\subset\left(\mathbb{R}\setminus\mathrm{Spec}(-\Delta)\right)\cup\set{0}$ and
\begin{align*}
    K\subset\set{h\in C^{1,\tau}\left(\Sigma\right): \text{h is a positive function}}
\end{align*}
where $0<\tau<1$. Then there exists a constant $C$ such that
\begin{align*}
    \norm{u}_{C^{1,\tau}\left(\Sigma\right)}\leq C
\end{align*}
for all solutions $u$ to \eqref{eq:gmf} with
\begin{align*}
    \rho\in I,\quad \alpha\in J,\quad h\in K.
\end{align*}
\end{theorem}
\begin{rem}
We give an example to show that the blow-up happens for $0<\rho<8\pi$ and $\alpha=\lambda_1\left(\Sigma\right)$. Let
$J_{\rho,\alpha}$ be defined as in \eqref{eq:functional}. Set $\alpha_n=\lambda_1\left(\Sigma\right)-1/n$ and let $u_n$ be a minimizer of $J_{\rho,\alpha_n}$ by the Trudinger-Moser inequality \eqref{eq:TM-1}. Then a straightforward calculation shows
\begin{align*}
    \lim_{n\to\infty}J_{\rho,\alpha_n}(u_n)=&\inf_{u\in H^1\left(\Sigma\right)}J_{\rho,\lambda_1\left(\Sigma\right)}(u)\leq\lim_{t\to+\infty}J_{\rho,\lambda_1\left(\Sigma\right)}\left(t\xi\right)=-\lim_{t\to+\infty}\ln\int_{\Sigma}he^{t\xi}=-\infty,
\end{align*}
where $\xi$ is a nonzero function which solves
\begin{align*}
    -\Delta\xi=\lambda_1\left(\Sigma\right)\xi.
\end{align*}
This implies that $\set{u_n}$ must be a blow-up sequence.
\end{rem}

According to \autoref{thm:compactness}, one can define the Leray-Schauder degree $d_{\rho,\alpha,h}$ for \eqref{eq:gmf} as follows (for more details about Leray-Schauder degree and its various properties we refer the reader to Chang \cite[Chapter 3]{Cha05methods}). Let
\begin{align*}
    X_{\tau}=\set{u\in C^{2,\tau}\left(\Sigma\right): \int_{\Sigma}u=0}.
\end{align*}
Clearly $X_{\tau}$, equipped with the $C^{2,\tau}\left(\Sigma\right)$ norm, is a Banach space. We introduce an operator $K_{\rho,\alpha, h}: X_{\tau}\To X_{\tau}$ by
\begin{align*}
    K_{\rho,\alpha, h}(u)=\left(-\Delta\right)^{-1}\left(\rho\left(\dfrac{he^{u}}{\int_{\Sigma}he^{u}}-1\right)+\alpha\left(u-\fint_{\Sigma}u\right)\right).
\end{align*}
The standard elliptic theory implies that $K_{\rho,\alpha, h}$ is a well defined compact operator. The generalized mean field equation \eqref{eq:gmf} is equivalent to $\left(1-K_{\rho,\alpha, h}\right)u=0$ in $X_{\tau}$. For any bounded open set $B\subset X_{\tau}$, the Leray-Schauder degree $\deg\left(1-K_{\rho,\alpha, h}, B, 0\right)$ is well defined provided $0\notin \left(1-K_{\rho,\alpha, h}\right)\left(\partial B\right)$. Let
\begin{align*}
    B_{R}^{X_{\tau}}=\set{u\in X_{\tau}:\norm{u}_{X_{\tau}}<R}
\end{align*}
be the ball in $X_{\tau}$.
Due to the compactness result in \autoref{thm:compactness}, we know that for all solutions $u$ to \eqref{eq:gmf}
\begin{align*}
    \norm{u-\ln\int_{\Sigma}he^{u}}_{X_{\tau}}\leq C,
\end{align*}
which implies that
\begin{align*}
   -C+\ln\int_{\Sigma}he^{u}\leq u\leq C+\ln\int_{\Sigma}he^{u}.
\end{align*}
If $\fint_{\Sigma}u=0$, then we obtain
\begin{align*}
    \abs{\ln\int_{\Sigma}he^{u}}\leq C.
\end{align*}
Consequently, the Leray-Schauder degree $\deg\left(1-K_{\rho,\alpha, h}, B_{R}^{X_{\tau}}, 0\right)$ is well defined for $R$ large and, in view of the homotopy invariance of the Leray-Schauder degree, is independent of $R$ as $R$ large. Thus \begin{align*}
    d_{\rho,\alpha}\coloneqq\lim_{R\to\infty}\deg\left(1-K_{\rho,\alpha, h}, B_{R}^{X_{\tau}}, 0\right)
\end{align*}
is well defined and is independent of $h$ due to the homotopy invariance.  In particular, according to Chen and Lin's result \cite[Theorem 2]{CheLin03topological},  we have the following
\begin{theorem}\label{thm:degree}
Let $(\Sigma,g)$ be a closed Riemann surface, $h$ a  positive smooth function,  $\lambda_1\left(\Sigma\right)$ the first positive eigenvalue of the negative Laplacian, $\rho$ and $\alpha$ real numbers. Then we have for $\rho\in\left(8k\pi, 8(k+1)\pi\right), k\in\mathbb{N}$ and $\alpha<\lambda_1\left(\Sigma\right)$,
\begin{align*}
    d_{\rho,\alpha}=d_{\rho,0}=\begin{cases}
        \dfrac{\left(k-\chi\left(\Sigma\right)\right)\dotsm \left(1-\chi\left(\Sigma\right)\right)}{k!},&k\in\mathbb{N}^{*},\\
        1,&k=0.
    \end{cases}
\end{align*}
\end{theorem}

An obvious consequence of \autoref{thm:degree} yields an existence result. Namely,
\begin{cor}\label{Cor3}
If $\chi\left(\Sigma\right)\leq0$ or $\chi(\Sigma)=2$ but $0\leq k\leq 1$, then there exists at least $\abs{d_{\rho,0}}$ solutions to \eqref{eq:gmf} provided that $\rho\in\mathbb{R}\setminus 8\pi\mathbb{N}^*$ and $\alpha<\lambda_1\left(\Sigma\right)$.
\end{cor}

We also consider the remaining cases of  \autoref{Cor3} and give the  existence result for \eqref{eq:gmf} on arbitrary closed Riemann surface. For this purpose, we  employ a minimax scheme. This is a standard method now,  which was used by Ding, Jost, Li and Wang \cite{DinJosLiWan99existence} to
study the mean field equation when $\rho\in(8\pi,16\pi)$, and by  Djadli \cite{Dja08existence} to solve the mean field equation when $\rho\in\left(8k\pi, 8(k+1)\pi\right)$ with $k\in\mathbb{N}^\ast$. Furthermore, Djadli and Malchiodi \cite{DjaMal08existence} employed it to
discuss the constant $Q$-curvature equation; Battaglia, Jevnikar and Malchiodi  \cite{BatJevMal15general} adopted it to study the Toda system.

Now we state our second main result as follows.
\begin{theorem}\label{Thm1}
Let $(\Sigma,g)$ be a closed Riemann surface, $h$ a  positive smooth function,  $\rho$ and $\alpha$ real numbers. Assume that $\rho\in (8k\pi, 8(k+1)\pi)$ for some non-negative integer number $k\in \mathbb{N}$ and $\lambda_1(\Sigma)$ is the first positive eigenvalue of the negative Laplacian.  Then there exists a solution to \eqref{eq:gmf}
provided $\alpha<\lambda_1\left(\Sigma\right)$.
\end{theorem}

\begin{rem}For $\rho\in (8k\pi, 8(k+1)\pi)$ and $\alpha<\lambda_1\left(\Sigma\right)$, note that $d_{\rho,\alpha}=0$ when $\Sigma$ is a $2$-sphere and $k\geq2$. 
Consequently, there exist at least two solutions in this case.
\end{rem}

Additionally, an analog of \autoref{Thm1} reads
\begin{rem}
Let $0=\lambda_0\left(\Sigma\right)<\lambda_1\left(\Sigma\right)<\dotsm$ be all distinct eigenvalues of $-\Delta$, $E_{\lambda_{k}\left(\Sigma\right)}$ the eigenfunction space with respect to $\lambda_{k}\left(\Sigma\right)$, and $E_l=E_{\lambda_0\left(\Sigma\right)}\oplus\dotsm\oplus E_{\lambda_{l}\left(\Sigma\right)}$. A similar argument above gives a critical point $u$ of $J_{\rho,\alpha}$ in $E_{l}^{\bot}$ provided $\rho\in\mathbb{R}\setminus 8\pi\mathbb{N}^*$ and $\alpha<\lambda_{l}\left(\Sigma\right)$, i.e.,
\begin{align*}
    -\Delta\left[u-\sum_{j=0}^L\hin{u}{\phi_j}\phi_j\right]=\rho\left[\dfrac{he^{u}}{\int_{\Sigma}he^{u}}-\sum_{j=0}^L\hin{\dfrac{he^{u}}{\int_{\Sigma}he^{u}}}{\phi_j}\phi_j\right]+\alpha\left[u-\sum_{j=0}^L\hin{u}{\phi_j}\phi_j\right],
\end{align*}
where $\set{\phi_0,\dotsc,\phi_L}$ is an orthonormal frame of $E_{l}$. This complements the results of Yang and Zhu \cite{YanZhu18existence}.
\end{rem}

The remaining part of this paper is organized as follows: firstly we study the compactness of the generalized mean field equation and prove \autoref{thm:compactness} in \autoref{sec:compactness}; secondly we give a new proof of a Trudinger-Moser inequality (cf. \autoref{thm:TM})  in \autoref{sec:TM}; finally we obtain the existence result for \eqref{eq:gmf} by using a minimax scheme and complete the proof of  \autoref{Thm1} in \autoref{sec:existence}. Hereafter we do not distinguish
sequence and subsequence; moreover, we often denote various constants by the same $C$.

\section{Compactness}\label{sec:compactness}

In this section, we first review some facts about the blow-up analysis for the mean field equation with positive potential, and then we prove the compactness result for  generalized mean field equation.

To study the general existence of the mean field equation \eqref{eq:mf} when $\rho\geq8\pi$, we consider the blow-up analysis of a sequence $u_n$ which solves
\begin{align}\label{eq:mf-sequence}
    -\Delta u_n=\rho_n\left(\dfrac{h_ne^{u_n}}{\int_{\Sigma}h_ne^{u_n}}-1\right),
\end{align}
where
\begin{align*}
    \rho_n\to\rho,\quad h_n\overset{C^{1}\left(\Sigma\right)}{\to}h.
\end{align*}
Up to adding a constant, we may assume
\begin{align*}
    \int_{\Sigma}h_ne^{u_n}=1.
\end{align*}
Since $\rho>0$ and $h$ is a positive smooth function, it yields that
\begin{align}\label{eq:energy}
    \int_{\Sigma}e^{u_n}\leq C.
\end{align}
Applying the Green  representation formula (cf. \cite[Theorem 4.13]{Aubin98some}) and the  potential estimate (cf. \cite[Lemma 7.12]{GilTru01elliptic}), we obtain
\begin{align*}
    \norm{u_n-\fint_{\Sigma}u_n}_{W^{1,p}\left(\Sigma\right)}\leq C_p\norm{\Delta u_n}_{L^1\left(\Sigma\right)}\leq C_p,\quad\forall p\in(1,2).
\end{align*}
If $\set{u_n^{+}}$ is bounded in $L^{\infty}\left(\Sigma\right)$, the standard elliptic estimate gives a uniform bound for $\set{u_n}$ in $L^{\infty}\left(\Sigma\right)$. In this case, $\set{u_n}$ is  compact in $C^2\left(\Sigma\right)$. If $\set{u_n}$ is a blow-up sequence, i.e.,
\begin{align*}
    \limsup_{n\to\infty}\max_{\Sigma}u_n=+\infty,
\end{align*}
we may assume $\rho_nh_ne^{u_n}\dif\mu_{\Sigma}$ converges to a nonzero Radon measure $\mu$. The singular set $S$ of the blow-up sequence $\set{u_n}$ is defined by
\begin{align*}
    S=\set{x\in\Sigma: \mu\left(\set{x}\right)\geq 4\pi}.
\end{align*}
It is clear that $S$ is a finite subset.
According to Brezis-Merle's estimate (\cite[Theorem 1]{BreMer91uniform}), for each smooth domain $\Omega\subset\Sigma$ and each solution $u$ to
\begin{align*}
\begin{cases}
    -\Delta u=f,&\text{in}\ \Omega,\\
    u=0,&\text{in}\ \partial\Omega,
\end{cases}
\end{align*}
where $f\in L^1\left(\Sigma\right)$, we have for every $\delta\in(0,4\pi)$
\begin{align*}
    \int_{\Omega}\exp\left(\dfrac{\left(4\pi-\delta\right)\abs{u}}{\norm{f}_{L^1\left(\Sigma\right)}}\right)\leq C_{\delta,\,\Omega}.
\end{align*}
Consequently, for every compact subset $K\subset\Sigma\setminus S$, there is a constant $C=C_K$ such that (cf. \cite[Lemma 2.8]{DinJosLiWan97differential})
\begin{align*}
    \norm{u_n-\fint_{\Sigma}u_n}_{L^{\infty}\left(K\right)}\leq C_K,
\end{align*}
which implies that $S$ is nonempty. One can check (cf. \citep[Page 1242-1243]{BreMer91uniform}) that $\lim_{n\to\infty}\fint_{\Sigma}u_n=-\infty$. Thus $\mu=\sum_{x\in\Sigma}\mu\left(\set{x}\right)\delta_{x}$. Moreover, the singular set can be characteristic of the blow-up set (\cite[Page 1240-1241]{BreMer91uniform}), i.e.,
\begin{align*}
    S=\set{x\in\Sigma: \exists \set{x_n}\subset\Sigma,\  \lim_{n\to\infty}x_n=x,\ \lim_{n\to\infty}u_n(x_n)=+\infty}.
\end{align*}
For $x_0\in S$, assume $S\cap B_{\delta}^{\Sigma}(x_0)=\set{x_0}$. Choose $B_{\delta}^{\Sigma}(x_0)\ni x_n\to x_0 $ such that
\begin{align*}
    \lambda_n\coloneqq u_n(x_n)=\max_{\overline{B_{\delta}^{\Sigma}(x_0)}}u_n\to+\infty.
\end{align*}
Without loss of generality, assume $B_{\delta}^{\Sigma}(x_0)$ is a Euclidean ball. Consider a rescaling
\begin{align*}
    \tilde u_n(x)=u_n\left(x_n+e^{-\lambda_n/2}x\right)-\lambda_n,\quad\abs{x}\leq e^{\lambda_n/2}\left(\delta-\abs{x_n}\right).
\end{align*}
From \eqref{eq:mf-sequence} and \eqref{eq:energy}, it follows  that
\begin{align*}
    -\Delta_{\mathbb{R}^2}\tilde u_n(x)=\rho_nh_n\left(x_n+e^{-\lambda_n/2}x\right) e^{\tilde u_n(x)}-\rho_n e^{-\lambda_n},
\end{align*}
and that
\begin{align*}
    \int_{\mathbb{B}_{R}}e^{\tilde u_n}\leq C_R.
\end{align*}
The above argument implies that $\tilde u_n$ converges strongly to $u_\infty$ in $H^2_{loc}\left(\mathbb{R}^2\right)$ as $n\to+\infty$, and then Chen-Li's classification result (\cite[Theorem 1]{CheLi91classification}) leads to
\begin{align*}
    u_\infty(x)=-2\ln\left(1+\dfrac{\rho h(x_0)}{8}\abs{x}^2\right).
\end{align*}
By Fatou's Lemma, we conclude that
\begin{align*}
    \mu\left(\set{x_0}\right)\geq 8\pi.
\end{align*}

Recall the following Pohozaev identity (cf. \cite[formula (8.1)]{KazWar74curvature}):
\begin{quotation}
    Assume $u$ is a solution to
    \begin{align*}
        -\Delta u=fe^{u}-P,\quad \text{in}\ \mathbb{B},
    \end{align*}
    where $f\in C^1\left(\mathbb{B}\right)$ and $P\in L^{\infty}\left(\mathbb{B}\right)$. Then for each $r\in(0,1)$ and every smooth function $F$,
    \begin{equation}\label{eq:pohozaev0}
    \begin{split}
        &\int_{\mathbf{B}_r}\hin{\nabla^2F-\dfrac12\Delta Fg}{\nabla u\otimes\nabla u}+\dfrac12\int_{\partial \mathbf{B}_r}\abs{\nabla u}^2\hin{\nabla F}{\nu}-\int_{\partial \mathbf{B}_r}\hin{\nabla F}{\nabla u}\hin{\nu}{\nabla u}\\
        =&\int_{\partial \mathbf{B}_r}fe^{u}\hin{\nabla F}{\nu}-\int_{\mathbf{B}_r}e^u\hin{\nabla f}{\nabla F}-\int_{\mathbf{B}_r}fe^{u}\Delta F-\int_{\mathbf{B}_r}P\hin{\nabla F}{\nabla u}.
    \end{split}
    \end{equation}
\end{quotation}
Take $F=\frac12\abs{x}^2$ in \eqref{eq:pohozaev0} to obtain
\begin{align}\label{eq:pohozaev1}
    \dfrac{r}{2}\int_{\partial \mathbf{B}_r}\abs{\nabla u}^2-r\int_{\partial \mathbf{B}_r}\hin{\nu}{\nabla u}^2=r\int_{\partial \mathbf{B}_r}fe^{u}-\int_{\mathbf{B}_r}e^u\hin{\nabla f}{x}-2\int_{\mathbf{B}_r}fe^{u}-\int_{\mathbf{B}_r}P\hin{x}{\nabla u}.
\end{align}
Applying the Pohozaev identity \eqref{eq:pohozaev1} to the blow-up sequence $u_n$, we conclude that
\begin{align*}
    \mu\left(\set{x_0}\right)=\lim_{r\to0}\left(\dfrac{r}{2}\int_{\partial\mathbf{B}_r}\hin{\nabla G}{\nu}^2-\dfrac{r}{4}\int_{\mathbf{B}_r}\abs{\nabla G}^2\right)=\dfrac{\mu\left(\set{x_0}\right)^2}{8\pi},
\end{align*}
where $G$ is the Green function satisfying
\begin{align*}
    -\Delta G(\cdot,y)=\sum_{x\in S}\mu\left(\set{x}\right)\delta_{x}-\rho,\quad\int_{\Sigma}G(\cdot,y)=0.
\end{align*}
Therefore, if $S$ has $k$ points then $\rho=8k\pi$. In other words, if $\rho\in\mathbb{R}\setminus 8\pi\mathbb{N}^*$, then there is a uniform constant $C$ such that
\begin{align*}
    \norm{u_n}_{L^{\infty}\left(\Sigma\right)}\leq C.
\end{align*}

We are in position to give a compactness result for the  generalized mean field equation as below.

\begin{proof}[Proof of \autoref{thm:compactness}]
We consider a sequence $u_n\in C^{2}\left(\Sigma\right)$ satisfying
\begin{align*}
    -\Delta u_n=\rho_n\left(h_ne^{u_n}-1\right)+\alpha_n\left(u_n-\fint_{\Sigma}u_n\right),
\end{align*}
where
\begin{align*}
    \rho_n\to \rho,\quad \alpha_n\to\alpha,\quad h_n\overset{C^{1,\tau}\left(\Sigma\right)}{\to}h,\quad\text{as}\ n\to\infty,
\end{align*}
and $0<\tau<1$. It suffices to prove that
\begin{align*}
    \norm{u_n}_{L^{\infty}\left(\Sigma\right)}\leq C.
\end{align*}
Arguing as in \eqref{eq:energy}, one can check that
\begin{align*}
    \int_{\Sigma}e^{u_n}\leq C.
\end{align*}

We claim the following potential estimate:
\begin{quotation}
    $\set{u_n-\fint_{\Sigma}u_n}$ is bounded in $W^{1,p}\left(\Sigma\right)$ for every $p\in(1,2)$.
\end{quotation}
In fact, for any fixed $p\in(0,1)$, applying the classical potential estimate for Laplacian, one has
\begin{align*}
    \norm{u_n-\fint_{\Sigma}u_n}_{W^{1,p}\left(\Sigma\right)}\leq C\norm{\Delta u_n}_{L^1\left(\Sigma\right)}\leq& C\left[\norm{\Delta u_n+\alpha_n\left(u_n-\fint_{\Sigma}u_n\right)}_{L^1\left(\Sigma\right)}+\norm{u_n-\fint_{\Sigma}u_n}_{L^1\left(\Sigma\right)}\right].
\end{align*}
Assume for some $C_n\to+\infty$,
\begin{align*}
    \norm{u_n-\fint_{\Sigma}u_n}_{W^{1,p}\left(\Sigma\right)}\geq C_n\norm{\Delta u_n+\alpha_n\left(u_n-\fint_{\Sigma}u_n\right)}_{L^1\left(\Sigma\right)}.
\end{align*}
Set $w_n=\frac{u_n-\fint_{\Sigma}u_n}{\norm{u_n-\fint_{\Sigma}u_n}_{W^{1,p}\left(\Sigma\right)}}$. Thus
\begin{align*}
    1=\norm{w_n}_{W^{1,p}\left(\Sigma\right)}\geq C_n\norm{\Delta w_n+\alpha_n w_n}_{L^1\left(\Sigma\right)},\quad\int_{\Sigma}w_n=0.
\end{align*}
We may assume $w_n$ converges to $w$ weakly in $W^{1,p}\left(\Sigma\right)$ and strongly in $L^p\left(\Sigma\right)$. Then
\begin{align*}
    \Delta w+\alpha w=0,\quad \int_{\Sigma}w=0.
\end{align*}
The classical potential estimate yields that
\begin{align*}
    1=\norm{w_n}_{W^{1,p}\left(\Sigma\right)}\leq C\left(\norm{\Delta w_n+\alpha_n w_n}_{L^1\left(\Sigma\right)}+\norm{w_n}_{L^1\left(\Sigma\right)}\right)\leq C\left(\norm{\Delta w_n+\alpha_n w_n}_{L^1\left(\Sigma\right)}+\norm{w_n}_{L^p\left(\Sigma\right)}\right).
\end{align*}
Letting $n\to\infty$, one finds
\begin{align*}
    1\leq C\norm{w}_{L^p\left(\Sigma\right)}.
\end{align*}
In particular, $\alpha$ is a nonzero eigenvalue of Laplacian which is a contradiction. Consequently,
\begin{align}\label{eq:w}
    \norm{u_n-\fint_{\Sigma}u_n}_{W^{1,p}\left(\Sigma\right)}\leq C\norm{\Delta u_n+\alpha_n\left(u_n-\fint_{\Sigma}u_n\right)}_{L^1\left(\Sigma\right)}\leq C.
\end{align}

Let $f_n$ be the solution of 
\begin{align}\label{eq:ff}
    -\Delta f_n=\alpha_n\left(u_n-\fint_{\Sigma}u_n\right),\quad\int_{\Sigma}f_n=0.
\end{align}
Thanks to \eqref{eq:w}, a standard elliptic estimate for \eqref{eq:ff} implies that
\begin{align*}
    \norm{f_n}_{W^{3,p}\left(\Sigma\right)}\leq C,\quad\forall p\in(1,2).
\end{align*}
Without loss of generality, we may assume $f_n$ converges to $f$ in $C^{1,\tau}\left(\Sigma\right)$ ($0<\tau<1$).
Define $\tilde u_n=u_n-f_n$ and $\tilde h_n=h_ne^{f_n}$. Then there holds
\begin{align*}
    -\Delta\tilde u_n=\rho_n\left(\tilde h_ne^{\tilde u_n}-1\right),
\end{align*}
and
\begin{align*}
    \tilde h_n\overset{C^{1,\tau}\left(\Sigma\right)}{\to}\tilde h\coloneqq he^{f}.
\end{align*}
Therefore, we may assume $\alpha_n=0$. The  blow-up analysis for the mean field equation with positive potential function is applicable. In particular, we complete the proof.
\end{proof}

\section{An improved Trudinger-Moser inequality}\label{sec:TM}

 From now on, we assume $\alpha<\alpha_1$. For $\rho\in\left(8k\pi, 8(k+1)\pi\right)(k\in \mathbb{N}^*)$, we shall adopt minimax arguments to derive the existence of solutions to \eqref{eq:gmf} on arbitrary closed Riemann surface.

As an application of  \autoref{thm:compactness}, we first prove the following Trudinger-Moser inequality, which was obtained by Yang \cite{Yan15extremal}.
\begin{theorem}\label{thm:TM}For $\alpha<\lambda_1$, there holds
\begin{align*}
    I_{8\pi,\alpha}(u)\coloneqq\dfrac{1}{16\pi}\int_{\Sigma}\left[\abs{\nabla u}^2-\alpha\left(u-\fint_{\Sigma}u\right)^2\right]+\fint_{\Sigma}u-\ln\int_{\Sigma}e^{u}\geq-C,\quad\forall u\in H^1\left(\Sigma\right).
\end{align*}
\end{theorem}
\begin{proof}
It is sufficient to consider the case $\alpha\in\left(0,\lambda_1\right)$.
We claim that for every $0<\varepsilon<16\pi$, there is a positive constant $C_{\varepsilon}$ such that
\begin{align}\label{eq:TM-weak}
    \dfrac{1}{16\pi-\varepsilon}\int_{\Sigma}\left[\abs{\nabla u}^2-\alpha\left(u-\fint_{\Sigma}u\right)^2\right]+\fint_{\Sigma}u-\ln\int_{\Sigma}e^{u}\geq-C_{\varepsilon},\quad\forall u\in H^1\left(\Sigma\right).
\end{align}
In fact, according to the classical Truding-Moser inequality \eqref{eq:TM}, we see that
\begin{align}\label{equ:en}\begin{split}
    \ln\int_{\Sigma}e^{u}\leq&\int_{\Sigma}e^{u^+}\\
    \leq&\dfrac{1}{16\pi}\int_{\Sigma}\abs{\nabla u^+}^2+\fint_{\Sigma}u^++C\\
    \leq&\dfrac{1}{16\pi}\int_{\Sigma}\left(\abs{\nabla u^+}^2-\alpha\left(u^+-\fint_{\Sigma}u^+\right)^2\right)+C\int_{\Sigma}\abs{u^+}^2+C.
    \end{split}
\end{align}
Without loss of generality, we may assume $\fint_{\Sigma}u=0$. For each $\eta\in(0,1)$, choose $a>0$ such that
\begin{align*}
    \abs{\set{x\in\Sigma: u(x)\geq a}}=\eta.
\end{align*}
Then it follows from \eqref{equ:en} that
\begin{align}\label{eq:ua}
\begin{split}
    \ln\int_{\Sigma}e^{u}=&a+\ln\int_{\Sigma}e^{u-a}\\
    \leq&\dfrac{1}{16\pi}\int_{\Sigma}\left(\abs{\nabla \left(u-a\right)^+}^2-\alpha\left(\left(u-a\right)^+-\fint_{\Sigma}\left(u-a\right)^+\right)^2\right)+C\int_{\Sigma}\abs{\left(u-a\right)^+}^2+C+a\\
    =&\dfrac{1}{16\pi}\int_{\Sigma}\left(\abs{\nabla  u}^2-\alpha u^2\right)-\dfrac{1}{16\pi}\int_{\Sigma}\left(\abs{\nabla \left(u-a\right)^-}^2-\alpha\left(\left(u-a\right)^--\fint_{\Sigma}\left(u-a\right)^-\right)^2\right)\\
    &+\dfrac{\alpha}{8\pi}\fint_{\Sigma}\left(u-a\right)^{+}\fint_{\Sigma}\left(u-a\right)^{-}+C\int_{\Sigma}\abs{\left(u-a\right)^+}^2+C+a\\
    \leq&\dfrac{1}{16\pi}\int_{\Sigma}\left(\abs{\nabla  u}^2-\alpha u^2\right)+\varepsilon\fint_{\Sigma}\abs{\left(u-a\right)^{-}}^2+C_{\varepsilon}\int_{\Sigma}\abs{\left(u-a\right)^+}^2+C+a
    \end{split}
\end{align}
On the other hand, Poincar\'e's  inequality and Kato's inequality yield that
\begin{align}\label{eq:uaa1}
    \norm{\left(u-a\right)^{+}}_{L^2\left(\Sigma\right)}\leq C\eta \norm{\nabla u}_{L^2\left(\Sigma\right)}\leq C\eta^{1/2}\left(\int_{\Sigma}\left(\abs{\nabla  u}^2-\alpha u^2\right)\right)^{1/2},
\end{align}
and that
\begin{align}\label{eq:uaa2}
       \norm{\left(u-a\right)^{-}}_{L^2\left(\Sigma\right)}\leq C\left(\int_{\Sigma}\left(\abs{\nabla  u}^2-\alpha u^2\right)\right)^{1/2}.
\end{align}
Insert \eqref{eq:uaa1} and \eqref{eq:uaa2} into \eqref{eq:ua}, therefore
\begin{align*}
    \ln\int_{\Sigma}e^{u}\leq \left(\dfrac{1}{16\pi}+\varepsilon+C_{\varepsilon}\eta\right)\int_{\Sigma}\left(\abs{\nabla  u}^2-\alpha u^2\right)+C+a.
\end{align*}
Notice that
\begin{align*}
    a\eta=a\int_{\set{u\geq a}}\leq\int_{\set{u\geq a}}u\leq\eta^{1/2}\norm{u}_{L^2\left(\Sigma\right)}\leq C\eta^{1/2}\left(\int_{\Sigma}\left(\abs{\nabla  u}^2-\alpha u^2\right)\right)^{1/2}.
\end{align*}
As a consequence,
\begin{align}\label{eq:ett}
    \ln\int_{\Sigma}e^{u}\leq \left(\dfrac{1}{16\pi}+\varepsilon+C_{\varepsilon}\eta\right)\int_{\Sigma}\left(\abs{\nabla  u}^2-\alpha u^2\right)+C+\dfrac{C_{\varepsilon}}{\eta}.
\end{align}
Let $\eta$ small enough. Finally \eqref{eq:TM-weak} follows from \eqref{eq:ett} immediately.

In view of  \eqref{eq:TM-weak}, for each $\rho_n=8\pi-1/n$,  there exits a minimizer $u_n\in H^1\left(\Sigma\right)$ to the functional
\begin{align}\label{equ:fun}
    I_{\rho_n,\alpha}(u)\coloneqq\dfrac{1}{2\rho_n}\left(\int_{\Sigma}\abs{\nabla u}^2-\alpha\left(u-\fint_{\Sigma}u\right)^2\right)+\fint_{\Sigma}u-\ln\int_{\Sigma}e^{u},\quad u\in H^1\left(\Sigma\right).
\end{align}
It is obvious to see
\begin{align*}
    \lim_{n\to\infty}I_{\rho_n,\alpha}\left(u_n\right)=\inf_{u\in H^1\left(\Sigma\right)}I_{8\pi,\alpha}(u).
\end{align*}
One can check that if $u_n$ blows up,  then (\cite[Theorem 1.1]{YanZhu18existence})
\begin{align}\label{eq:lower-I}
    \inf_{u\in H^1\left(\Sigma\right)}I_{8\pi,\alpha}(u)=-1-\ln\pi-\frac{1}{2}\max_{\Sigma}A,
\end{align}
where $A$ is the regular part of the Green fucntion $G$, i.e.,
\begin{align}\label{eq:green}
    -\Delta G(\cdot,y)=8\pi\left(\delta_{y}-1\right)+\alpha G(\cdot,y),\quad\int_{\Sigma}G(\cdot,y)=0,
\end{align}
Here $G$ takes the form
\begin{align}\label{eq:green2}
    G(x,x_0)=-4\ln r(x)+A(x_0)+\psi(x),
\end{align}
where $r$ denotes the geodesic distance between $x$ and $x_0$, $\psi(x)\in C^1(\Sigma)$ and $\psi(x_0)=0$. We next give a new proof of  \eqref{eq:lower-I}. Without loss of generality, we may assume $u_n$ satisfies
\begin{align*}
    -\Delta u_n=\rho_n\left(e^{u_n}-1\right)+\alpha\left(u_n-\fint_{\Sigma}u_n\right)
\end{align*}
and $u_n-\fint_{\Sigma}u_n$ converges to the Green function $G$ weakly in $W^{1,p}\left(\Sigma\right)$ and strongly in $L^p\left(\Sigma\right)$. Let $\set{x_0}$ be the blow-up point. Let $f_n$ be given by
\begin{align*}
    -\Delta f_n=\alpha\left(u_n-\fint_{\Sigma}u_n\right),\quad\int_{\Sigma}f_n=0.
\end{align*}
Due to the potential estimate \eqref{eq:w}, we obtain that $f_n$ is compact in  $C^1\left(\Sigma\right)$. Set $w_n=u_n-f_n$ and $k_n=e^{f_n}$. We may assume $k_n=e^{f_n}$ converges to $k\coloneqq e^{f}$ in $C^1\left(\Sigma\right)$ as $n\to\infty$. And then, there holds
\begin{align*}
    -\Delta\left(G(\cdot,x_0)-f\right)=8\pi\left(\delta_{x_0}-1\right).
\end{align*}
Observe that
\begin{align*}
    -\Delta w_n=\rho_n\left(k_ne^{w_n}-1\right).
\end{align*}
It is easy to check that $\set{w_n}$ is a blow-up sequence. Since
\begin{align*}
    I_{\rho_n,\alpha}(u_n)=\left[\dfrac{1}{2\rho_n}\int_{\Sigma}\abs{\nabla w_n}^2+\fint_{\Sigma}w_n-\ln\int_{\Sigma}k_ne^{w_n}\right]+\left[\dfrac{\alpha}{2\rho_n}\int_{\Sigma}\left(u_n-\fint_{\Sigma}u_n\right)^2-\dfrac{1}{2\rho_n}\int_{\Sigma}\abs{\nabla f_n}^2\right],
\end{align*}
we conclude from the proof of \cite[formula (1.4)]{DinJosLiWan97differential} that
\begin{align}\label{eq:ii}
\begin{split}
    \lim_{n\to\infty}I_{\rho_n,\alpha}(u_n)=&\lim_{n\to\infty}\left[\dfrac{1}{2\rho_n}\int_{\Sigma}\abs{\nabla w_n}^2+\fint_{\Sigma}w_n-\ln\int_{\Sigma}k_ne^{w_n}\right]+\lim_{n\to\infty}\left[\dfrac{\alpha}{2\rho_n}\int_{\Sigma}\left(u_n-\fint_{\Sigma}u_n\right)^2-\dfrac{1}{2\rho_n}\int_{\Sigma}\abs{\nabla f_n}^2\right]\\
    \geq&-1-\ln\pi-\left(\ln k(x_0)+\dfrac12\tilde A(x_0)\right)+\dfrac{\alpha}{16\pi}\int_{\Sigma}G(\cdot,x_0)^2-\dfrac{1}{16\pi}\int_{\Sigma}\abs{\nabla f}^2.
    \end{split}
\end{align}
Here $\tilde A$ is the regular part of the Green function $\tilde G$, i.e.,
\begin{align*}
    -\Delta\tilde G(\cdot,y)=8\pi\left(\delta_y-1\right),\quad\int_{\Sigma}\tilde G(\cdot,y)=0.
\end{align*}
and in a local normal coordinates $x$ centering at $x_0$, 
\begin{align}\label{eq:lnx}
    \tilde G(x,x_0)=-4\ln\abs{x}+\tilde A(x_0)+\tilde\psi(x),
\end{align}
where $\tilde\psi$ is a smooth function with $\tilde\psi(x_0)=0$. 
Note that $G(\cdot,x_0)=\tilde G(\cdot,x_0)+f$. Together with \eqref{eq:green2} and \eqref{eq:lnx}, it leads to  $\tilde A(x_0)=A(x_0)-f(x_0)$ and
\begin{align}\label{eq:fff}
    \dfrac{1}{16\pi}\int_{\Sigma}\abs{\nabla f}^2=-\dfrac{1}{16\pi}\int_{\Sigma}f\Delta f=\dfrac{\alpha}{16\pi}\int_{\Sigma}fG=\dfrac{1}{16\pi}\int_{\Sigma}f\left(-\Delta G+8\pi\left(1-\delta_{x_0}\right)\right)=-\dfrac{1}{16\pi}\int_{\Sigma}\Delta f G-\dfrac{f(x_0)}{2}.
\end{align}
Combing \eqref{eq:ii} and  \eqref{eq:fff}, one has
\begin{align*}
    \lim_{n\to\infty}I_{\rho_n,\alpha}(u_n)\geq&-1-\ln\pi-\left(f(x_0)+\dfrac12\tilde A(x_0)\right)+\dfrac{f(x_0)}{2}=-1-\ln\pi-\dfrac12A(x_0)\geq-1-\ln\pi-\dfrac12\max_{\Sigma}A.
\end{align*}
On the other hand, we construct a sequence of functions $\set{\phi_n}$ (which is in \cite{YanZhu18existence}) satisfying
\begin{align*}
\lim_{n\to\infty}I_{8\pi,\alpha}\left(\phi_n-\fint_{\Sigma}\phi_n\right)
=-1-\ln\pi-\dfrac{1}{2}\max_{\Sigma}A.
\end{align*}
Precisely, suppose that $A(x_0)=\max_{x\in \Sigma}A_x$. Let $r=r(x)$ be the geodesic distance between $x$ and $x_0$. Set
\begin{align*}
\phi_n(x)=\begin{cases}
    c-2\log \left(1+\frac{r^2}{8r_n^2}\right),& x\in B_{Rr_n}^{\Sigma}(x_0),\\
G(x,x_0)-\eta(x)\psi(x),& x\in B_{2Rr_n}^{\Sigma}(x_0)\backslash B_{Rr_n}^{\Sigma}(x_0),\\
G(x,x_0), & x\in \Sigma\backslash B_{2Rr_n}^{\Sigma}({x_0}),
\end{cases}
\end{align*}
where $r_n=e^{-\lambda_n/2},  \lambda_n=\max_{\Sigma}u_n, \eta \in C^\infty_0\left(B^{\Sigma}_{2Rr_n}({x_0})\right)$ is a cut-off function satisfying $\eta\equiv1$ in $B^{\Sigma}_{Rr_n}({x_0})$ and $\abs{\nabla \eta(x)}\leq \frac{4}{Rr_n}$ for $x\in B^{\Sigma}_{2Rr_n}(x_0)$,  $G(\cdot,x_0)$ and  $\psi$ are defined in \eqref{eq:green}, \eqref{eq:green2} respectively. Moreover
\begin{align*}
c=2\ln\left(1+R^2/8\right)-4\ln R-4\ln r_n+A(x_0).
\end{align*}
Then we estimate the three terms of \eqref{equ:fun} respectively. By a straightforward calculation as in the proof of (\cite{YanZhu18existence}, Theorem 1.1.), we can get the desired result.
Thus, we end the proof of \eqref{eq:lower-I} and complete the proof the theorem.

\end{proof}

We next prove an improved Trudinger-Moser inequality which will be used in the next section. Namely,
\begin{lem}
Assume $\alpha<\lambda_1\left(\Sigma\right)$. For positive integer number $k$, positive numbers $\delta_0,\gamma_0$, consider
\begin{align*}
    X_{k,\delta_0,\gamma_0}\coloneqq \set{u\in H^1\left(\Sigma\right): \exists\ \text{subdomains}\  \Omega_i\subset\Sigma\ \text{s.t.}\ \min_{1\leq i<j\leq k}\mathrm{dist}\left(\Omega_i,\Omega_j\right)\geq\delta_0\ \text{and}\ \min_{1\leq i\leq k}\int_{\Omega_i}e^{u}\geq\gamma_0\int_{\Sigma}e^{u}}.
\end{align*}
For every $\varepsilon\in(0,16k\pi)$, there exists a positive constant $C=C_{k,\delta_0,\gamma_0,\varepsilon}$ such that
\begin{align}\label{eq:gtm}
    \ln\int_{\Sigma}e^{u}\leq\dfrac{1}{16k\pi-\varepsilon}\int_{\Sigma}\left(\abs{\nabla u}^2-\alpha\left(u-\fint_{\Sigma}u\right)^2\right)+\fint_{\Sigma}u+C,\quad\forall u\in X_{k,\delta_0,\gamma_0}.
\end{align}
\end{lem}
\begin{proof}
Given $u\in X_{k,\delta_0,\gamma_0}$ with $\fint_{\Sigma}u=0$, applying an improved Trudinger-Moser inequality for $\alpha=0$ (cf. \cite[Theorem 2.1]{CheLi91prescribing}), we obtain
\begin{align*}
    \ln\int_{\Sigma}e^{u}\leq\dfrac{1}{16k\pi-\varepsilon}\int_{\Sigma}\abs{\nabla u}^2+\fint_{\Sigma}u+C,\quad\forall u\in X_{k,\delta_0,\gamma_0}.
\end{align*}
Using a similar method as in  the proof of \eqref{eq:TM-weak}, we complete the proof.
\end{proof}

\section{Existence}\label{sec:existence}

The remaining part of this section is devoted to the proof of  \autoref{Thm1}. We begin with the topological structure for the formal set of barycenters, which is to be used in the minimax argument.
Denote $\mathcal{D}\left(\Sigma\right)$ by the distributions on $\Sigma$. We will use on $\mathcal{D}\left(\Sigma\right)$ the metric given by $C^1\left(\Sigma\right)^*$ and which will be denoted by $\mathbf{d}(\cdot,\cdot)$. In other words, for $\sigma,\zeta\in\mathcal{D}(\Sigma)$,
\begin{align*}
    \mathbf{d}(\sigma,\zeta)=\sup_{\psi\in C^1\left(\Sigma\right): \norm{\psi}_{C^1\left(\Sigma\right)}\leq 1}\hin{\sigma-\zeta}{\psi}.
\end{align*}
We consider
\begin{align*}
    \Sigma_k=\set{\sum_{i=1}^kt_i\delta_{x_i}: t_i\geq0, x_i\in\Sigma,  \sum_{i=1}^kt_i=1}\subset\mathcal{D}\left(\Sigma\right)
\end{align*}
which is known as the formal set of barycenters of $\Sigma$ of order $k$. It is the fact that $\Sigma_k$ is non-contractible for every $k\geq1$ (cf. \cite[Lemma 4.7]{Dja08existence}).

The first step is to construct a continuous projection from low sublevels of the functional to the $k$-th barycenters, precisely
\begin{lemma}\label{prop:Psi}
Let $\rho\in\left(8k\pi, 8(k+1)\pi\right), k\in\mathbb{N}$ and $\alpha<\lambda_1\left(\Sigma\right)$. Then for $L$ sufficiently large there exists a continuous projection
\begin{align*}
    \Psi: J_{\rho,\alpha}^{-L}\coloneqq\set{u\in H^1\left(\Sigma\right): J_{\rho,\alpha}(u)\leq -L}\To\Sigma_{k}.
\end{align*}
Moreover, if $\frac{e^{u_n}}{\int_{\Sigma}e^{u_n}}\dif\mu_{\Sigma}$ converges to $\sigma\in\Sigma_{k}$, then $\Psi(u_n)\to\sigma$.
\end{lemma}
\begin{proof}The proof is similar to  \cite[Proposition 4.1]{MarLopRui18compactness}.
There exists a large $L$ such that for every $u_n\in J_{\rho,\alpha}^{-L}$, the following holds
\begin{quotation}
    If $J_{\rho,\alpha}(u_n)\to-\infty$, then up to a subsequence,
\begin{align*}
    \sigma_n\coloneqq\dfrac{e^{u_n}}{\int_{\Sigma}e^{u_n}}\dif\mu_{\Sigma}\to\sigma\in\Sigma_k.
\end{align*}
\end{quotation}
One can prove this fact by contradiction that there exists $k+1$ points $\set{x_1,\dotsc,x_{k+1}}\subset\mathrm{supp}(\sigma)$. Take $r>0$ such that $B^{\Sigma}_{2r}(x_i)\cap B^{\Sigma}_{2r}(x_j)=\emptyset$ for all $i\neq j$. There exists $\varepsilon>0$ such that $\sigma\left(B^{\Sigma}_{r}(x_i)\right)>2\varepsilon$ for all $i$. Thus we may assume
\begin{align*}
    \dfrac{\int_{B^{\Sigma}_{r}(x_i)}e^{u_n}}{\int_{\Sigma}e^{u_n}}\geq\varepsilon,\quad i=1,\dotsc, k+1.
\end{align*}
The improved Trudinger-Moser inequality \eqref{eq:gtm} implies that for any $\tilde\varepsilon\in(0,16(k+1)\pi)$, there exists a constant $C=C(\tilde\varepsilon,\varepsilon,r)$ such that
\begin{align}\label{equ:ru}
\begin{split}
    \ln\int_{\Sigma}he^{u_n}\leq\ln\int_{\Sigma}e^{u_n}+C\leq \dfrac{1}{16(k+1)\pi-\tilde\varepsilon}\int_{\Sigma}\left(\abs{\nabla u_n}^2-\alpha\left(u_n-\fint_{\Sigma}u_n\right)^2\right)+\fint_{\Sigma}u_n+C.
    \end{split}
\end{align}
Taking $\tilde\varepsilon$ small, we conclude from \eqref{equ:ru} that
\begin{align*}
    J_{\rho,\alpha}(u_n)\geq-C,
\end{align*}
which yields a contradiction.

Therefore, for any $\varepsilon_0>0$, there exists $L_0$ large enough such that if $L>L_0$, then
\begin{align*}
    \dfrac{e^{u}}{\int_{\Sigma}e^{u}}\dif\mu_{\Sigma}\in\set{\sigma\in\mathcal{D}(\Sigma): \mathbf{d}\left(\sigma,\Sigma_k\right)<\varepsilon_0},\quad \forall u\in J_{\rho,\alpha}^{-L}.
\end{align*}
Notice that for small $\varepsilon_0$, there is a continuous retraction (cf. \cite[Proposition 2.2]{BatJevMal15general})
\begin{align*}
    \psi_k: \set{\sigma\in\mathcal{D}(\Sigma): \mathbf{d}\left(\sigma,\Sigma_k\right)<\varepsilon_0}\To\Sigma_k.
\end{align*}
Now define $\Psi$ by
\begin{align*}
    \Psi: u\mapsto\psi_k\left(\dfrac{e^{u}}{\int_{\Sigma}e^{u}}\dif\mu_{\Sigma}\right)
\end{align*}
 to finish the proof.
\end{proof}
Next, we aim to construct a continuous map $\Phi_{\lambda}:\Sigma_k\To J_{\rho,\alpha}^{-L}$ such that $\Psi\circ\Phi_{\lambda}$ is homotopic to the identity map for $\lambda$ large.
For this purpose, we choose a smooth non-decreasing function $\eta:\mathbb{R}\To\mathbb{R}$ such that
\begin{align*}
    \eta(t)=\begin{cases}
        t,& t\leq 1,\\
        2,& t\geq 2,
    \end{cases}
\end{align*}
and denote $\eta_{\delta}(t)=\delta\eta\left(t/\delta\right)$ for small $\delta>0$. Given $\sigma=\sum_{i=1}^kt_i\delta_{x_i}\in\Sigma_{k}$ and $\lambda>0$, we define
\begin{align*}
    \phi_{\lambda,\sigma}(x)=\ln\left(\sum_{i=1}^kt_i\dfrac{8\lambda^2}{\left(1+\lambda^2\eta_{\delta}^2\left(\mathrm{dist}^{\Sigma}(x,x_i)\right)\right)^2}\right).
\end{align*}
Then the following lemma holds.

\begin{lem}\label{lem:Phi}Let $\rho\in\left(8k\pi, 8(k+1)\pi\right), k\in\mathbb{N}$ and $\alpha<\lambda_1\left(\Sigma\right)$. We can choose a small $\delta$ and large $\lambda$ such that for any $\sigma\in\Sigma_k$,
\begin{align*}
    J_{\rho,\alpha}\left(\phi_{\lambda,\sigma}\right)\leq-\left(1-\dfrac{8k\pi}{\rho}\right)\ln\lambda,
\end{align*}
and
\begin{align*}
    \dfrac{e^{\phi_{\lambda,\sigma}}}{\int_{\Sigma}e^{\phi_{\lambda,\sigma}}}\dif\mu_{\Sigma}\to\sigma,\quad\text{as}\ \lambda\to+\infty.
\end{align*}
\end{lem}
\begin{proof}
Assume  $\delta\ll1$ and $\lambda\gg1$. Set $\sigma=\sum_{i=1}^kt_i\delta_{x_i}, t_i\geq0, \sum_{i=1}^kt_i=1$. By definition
\begin{align*}
    \phi_{\lambda,\sigma}(x)=\begin{cases}\ln\left(\frac{8\lambda^2}{\left(1+4\delta^2\lambda^2\right)^2}\right),
        &x\in\Sigma\setminus\left(\cup_{i=1}^kB^{\Sigma}_{2\delta}(x_i)\right),\\
        \ln\left(\frac{8\lambda^2t_i}{\left(1+\lambda^2\eta_{\delta}^2\left(\mathrm{dist}^{\Sigma}(x,x_i)\right)\right)^2}+\frac{8\lambda^2(1-t_i)}{\left(1+4\delta^2\lambda^2\right)^2}\right),&x\in B^{\Sigma}_{2\delta}(x_i).
    \end{cases}
\end{align*}
Choose local normal coordinates $x$ around $x_i$. It is well known that
\begin{align*}
    \det\left(g_{ij}(x)\right)=1+O\left(\abs{x-x_i}^2\right).
\end{align*}
For any fixed $\delta>0$, and every smooth function $\psi$ on $\Sigma$, it holds
\begin{align}\label{equ:si}
\begin{split}
   \int_{\Sigma}e^{\phi_{\lambda,\sigma}}\psi=&\sum_{i=1}^kt_i\int_{B^{\Sigma}_{r}(x_i)}\dfrac{8\lambda^2}{\left(1+\lambda^2\left(\mathrm{dist}^{\Sigma}(x,x_i)\right)^2\right)^2}+O\left(r^{-4}\lambda^{-2}\right)\\
   =&\left(1+O\left(r^2\right)\right)\sum_{i=1}^kt_i\int_{\mathbf{B}_{r}(x_i)}\dfrac{8\lambda^2}{\left(1+\lambda^2\abs{x-x_i}^2\right)^2}\psi(x)+O\left(r^{-4}\lambda^{-2}\right)\\
   =&\left(1+O\left(r^2\right)\right)\sum_{i=1}^kt_i\int_{\mathbf{B}_{r}(x_i)}\dfrac{8\lambda^2}{\left(1+\lambda^2\abs{x-x_i}^2\right)^2}\left(\psi(x)-\psi(x_i)\right)+8\pi\sum_{i=1}^kt_i\psi(x_i)+O\left(r^2\right)+O\left(r^{-4}\lambda^{-2}\right)\\
   =&8\pi\sum_{i=1}^kt_i\psi(x_i)+O\left(r\right)+O\left(r^{-4}\lambda^{-2}\right).
   \end{split}
\end{align}
Passing to the limit $\lambda\to+\infty$ first, and then  $r\to0$, we have by \eqref{equ:si} that
\begin{align*}
    \lim_{\lambda\to+\infty}\int_{\Sigma}e^{\phi_{\lambda,\sigma}}=8\pi
\end{align*}
and that
\begin{align*}
    \lim_{\lambda\to+\infty}\int_{\Sigma}\dfrac{e^{\phi_{\lambda,\sigma}}}{\int_{\Sigma}e^{\phi_{\lambda,\sigma}}}\psi=\sum_{i=1}^kt_i\psi(x_i)=\hin{\sigma}{\psi}.
\end{align*}
Thus we confirm the second claim.

For the first assertion, we compute
\begin{align}\label{equ:si1}
\begin{split}
    \int_{\Sigma}\abs{\nabla\phi_{\lambda,\sigma}}^2=&\sum_{i=1}^k\int_{B_{2\delta}^{\Sigma}(x_i)}\left(\dfrac{\frac{8\lambda^2t_i}{\left(1+\lambda^2\eta_{\delta}^2\right)^2}}{\frac{8\lambda^2t_i}{\left(1+\lambda^2\eta_{\delta}^2\right)^2}+\frac{8\lambda^2(1-t_i)}{\left(1+4\delta^2\lambda^2\right)^2}}\right)^2\cdot\left(\dfrac{4\lambda^2\eta_{\delta}\eta_{\delta}'}{1+\lambda^2\eta_{\delta}^2}\right)^2\\
    \leq&\sum_{i=1}^k\int_{B_{\delta}^{\Sigma}(x_i)}\left(\dfrac{4\lambda^2\mathrm{dist}^{\Sigma}\left(x,x_i\right)}{1+\lambda^2\left(\mathrm{dist}^{\Sigma}\left(x,x_i\right)\right)^2}\right)^2+C\\
    =&\sum_{i=1}^k\left(1+O\left(\delta^2\right)\right)\int_{\mathbf{B}_{\delta}}\left(\dfrac{4\lambda^2\abs{x}}{1+\lambda^2\abs{x}^2}\right)^2+C\\
    \leq&32k\pi\left(1+O\left(\delta^2\right)\right)\ln\lambda+C.
    \end{split}
\end{align}
This together with Poincar\'e's inequality leads to
\begin{align}\label{equ:si2}
\begin{split}
    \int_{\Sigma}\left(\phi_{\lambda,\sigma}-\fint_{\Sigma}\phi_{\lambda,\sigma}\right)^2\leq&\int_{\Sigma}\left(\phi_{\lambda,\sigma}-\ln\left(\dfrac{8\lambda^2}{\left(1+4\delta^2\lambda^2\right)^2}\right)\right)^2\\
    =&\sum_{i=1}^k\int_{B_{2\delta}^{\Sigma}(x_i)}\left(\phi_{\lambda,\sigma}-\ln\left(\frac{8\lambda^2}{\left(1+4\delta^2\lambda^2\right)^2}\right)\right)^2\\
    \leq&C\delta^2\int_{B_{2\delta}^{\Sigma}(x_i)}\abs{\nabla\phi_{\lambda,\sigma}}^2\\
    =&O\left(\delta^2\right)\ln\lambda.
    \end{split}
\end{align}
 By \eqref{equ:si1} and \eqref{equ:si2},  it follows from Jensen's inequality that
\begin{align}\label{equ:si3}
\begin{split}
    J_{\rho,\alpha}\left(\phi_{\lambda,\sigma}\right)\leq&\dfrac{1}{2\rho}\int_{\Sigma}\abs{\nabla \phi_{\lambda,\sigma}}^2+\fint_{\Sigma}\phi_{\lambda,\sigma}-\ln\int_{\Sigma}e^{\phi_{\lambda,\sigma}}+C\int_{\Sigma}\left(\phi_{\lambda,\sigma}-\fint_{\Sigma}\phi_{\lambda,\sigma}\right)^2+C\\
    \leq&\dfrac{1}{2\rho}\int_{\Sigma}\abs{\nabla \phi_{\lambda,\sigma}}^2+C\int_{\Sigma}\left(\phi_{\lambda,\sigma}-\fint_{\Sigma}\phi_{\lambda,\sigma}\right)^2+C\\
    \leq&\left(\dfrac{16k\pi}{\rho}-2+O\left(\delta^2\right)\right)\ln\lambda+C.
    \end{split}
\end{align}
Since $\rho>8k\pi$, choosing $\delta$ small and $\lambda$ large, by \eqref{equ:si3}, we obtain 
\begin{align*}
    J_{\rho,\alpha}\left(\phi_{\lambda,\sigma}\right)\leq-\left(1-\dfrac{8k\pi}{\rho}\right)\ln\lambda.
\end{align*}
This ends the proof of the lemma.
\end{proof}

Define $\Phi_{\lambda}(\sigma)=\phi_{\lambda,\sigma}$.
According to \autoref{prop:Psi} and \autoref{lem:Phi}, we conclude that for large $L$ and $\lambda\geq\lambda_L=e^{L/(1-8k\pi/\rho)}$
\begin{align*}
    \Sigma_k\overset{\Phi_{\lambda}}{\To}J_{\rho,\alpha}^{-\left(1-8k\pi/\rho\right)\ln\lambda}\overset{\Psi}{\To}\Sigma_k
\end{align*}
and $\lim_{\lambda\to+\infty}\Psi\circ\Phi_{\lambda}=\mathrm{Id}$. In particular, $\Psi\circ\Phi_{\lambda}$ is homotopic to the identity on $\Sigma_k$ provided $\lambda\geq\lambda_L$.

\hspace{2cm}\\
\indent Now we can prove the existence result by employing a minimax method.

\begin{proof}[Proof of \autoref{Thm1}]
Let $\hat\Sigma_k=\Sigma_k\times[0,1]/\Sigma\times\set{0}$ denote the cone over $\Sigma_k$. Choose $L$ large and consider the following class
\begin{align*}
    \Gamma_{\lambda}=\Gamma_{\lambda,\rho}=\set{\gamma\in C^0\left(\hat\Sigma_k, H^1\left(\Sigma\right)\right): \gamma(\cdot\times\set{1})=\Phi_{\lambda}},\quad\forall \lambda\geq e^{L/(1-8k\pi/\rho)}.
\end{align*}
It is clear that $\Gamma_{\lambda}$ is nonempty. Set
\begin{align*}
    \bar\Gamma_{\lambda}=\bar\Gamma_{\lambda,\rho}=\inf_{\gamma\in\Gamma_{\lambda_M}}\sup_{z\in\hat\Sigma_k}J_{\rho,\alpha}\left(\gamma(z)\right)\leq-\left(1-\dfrac{8k\pi}{\rho}\right)\ln\lambda.
\end{align*}
Using the fact that $\Sigma_k$ is non-contractible and $\Psi\circ\Phi_{\lambda}$ is homotopic to the identity, we have
\begin{align*}
    \bar\Gamma_{\lambda}>-2\left(1-\dfrac{8k\pi}{\rho}\right)\ln\lambda.
\end{align*}
For otherwise, there exists $\gamma\in\Gamma_{\lambda}$ with
\begin{align*}
    \sup_{z\in\hat\Sigma_k}J_{\rho,\alpha}\left(\gamma(z)\right)\leq-\dfrac{3}{2}\left(1-\dfrac{8k\pi}{\rho}\right)\ln\lambda.
\end{align*}
Write $z=(y,t)$ with $y\in\Sigma_k$, then the map
\begin{align*}
    t\mapsto\Psi\circ\gamma(\cdot,t)
\end{align*}
gives a homotopy in $\Sigma_k$ between $\Psi\circ\Phi_{\lambda}=\Psi\circ\gamma(\cdot,1)$ and a constant map $\Psi\circ\gamma(\cdot,0)$. But this is impossible since $\Sigma_k$ is non-contractible and $\Psi\circ\Phi_{\lambda}$ is homotopic to the identity.

By the monotonicity trick (cf. \cite[Lemma 5.1]{BatJevMal15general}), there exists $\Lambda\subset\left(1-\mu_0,1+\mu_0\right)$ such that $\Lambda$ is dense in $\left[1-\mu_0,1+\mu_0\right]$. Moreover, for any $\mu\in\Lambda$, the functional $J_{\mu\rho,\alpha}$ posses a bounded Palais-Smale sequence $\set{u_n}$ at level $\bar\Gamma_{\lambda,\mu\rho}$. Standard arguments show that there is a critical point $u_{n}$ of $J_{\mu_n\rho,\alpha}$ for each $\mu_n\in\Lambda$. Then applying the compactness result in \autoref{thm:compactness} and the denseness of $\Lambda$, we obtain a critical point of $J_{\rho,\alpha}$ and hence complete the proof.
\end{proof}


\end{document}